\documentclass[12pt,reqno]{article}

\usepackage[usenames]{color}
\usepackage{amsmath}
\usepackage{amsthm}
\usepackage{amssymb}
\usepackage{amsfonts}
\usepackage{amscd}
\usepackage{mathtools}
\usepackage{ytableau}
\usepackage{subcaption}
\usepackage[inline]{enumitem}
\DeclareRobustCommand{\ustirlingfirst}{\genfrac[]{0pt}{}}
\DeclareRobustCommand{\rawfrac}{\genfrac{}{}{0pt}{}}
\DeclarePairedDelimiter{\insideBraces}{\{}{\}}
\DeclarePairedDelimiterX\outsideBraces[1]{\{}{\}}{\mskip-8mu\insideBraces*{#1}\mskip-8mu}
\newcommand{\Astirling}[2]{\outsideBraces*{\rawfrac{#1}{#2}}}
\newcommand{\stirling}[2]{\insideBraces*{\rawfrac{#1}{#2}}}

\usepackage[colorlinks=true,
linkcolor=webgreen,
filecolor=webbrown,
citecolor=webgreen]{hyperref}

\definecolor{webgreen}{rgb}{0,.5,0}
\definecolor{webbrown}{rgb}{.6,0,0}

\usepackage{color}
\usepackage{fullpage}
\usepackage{float}

\usepackage{latexsym}
\usepackage{epsf}
\usepackage{breakurl}

\newcommand{\seqnum}[1]{\href{https://oeis.org/#1}{\rm \underline{#1}}}
\newcommand{\arxiv}[1]{\url{https://arxiv.org/abs/#1}}

\begin{document}
    \theoremstyle{plain}
    \newtheorem{theorem}{Theorem}

    \newtheorem{corollary}[theorem]{Corollary}
    \newtheorem{lemma}[theorem]{Lemma}
    \newtheorem{proposition}[theorem]{Proposition}

    \theoremstyle{definition}
    \newtheorem{definition}[theorem]{Definition}
    \newtheorem{example}[theorem]{Example}
    \newtheorem{conjecture}[theorem]{Conjecture}

    \theoremstyle{remark}
    \newtheorem{remark}[theorem]{Remark}
    \numberwithin{theorem}{section}
    \begin{center}
        \vskip 1cm{\LARGE\bf {Counting Nonattacking Chess Piece Placements: Bishops and Anassas}
        }
        \vskip 1cm
        \large
        Eder G. Santos\\
        Independent Researcher\\
	Brazil \\
        \href{mailto:email}{\tt ederguidsa@gmail.com} \\
    \end{center}

    \vskip .2 in
    \begin{abstract}
        We derive recurrences and closed-form expressions for counting nonattacking placements of two types of chess pieces with unbounded straight-line moves, namely the bishop (two diagonal moves) and the anassa (one horizontal or vertical move and one diagonal move), placed on a standard square chessboard. Additionally, we obtain explicit expressions for the corresponding quasi-polynomial coefficients. The recurrences are derived by analyzing how nonattacking configurations attack a specific subset of board squares, employing a bijective argument to establish the relations. The main results are simplifications of known expressions for the bishop and a general counting formula for the anassa.
    \end{abstract}

\section{Introduction}
In a series of papers, Chaiken, Hanusa, and Zaslavsky \cite{chaiken1, chaiken2, chaiken3, chaiken4, chaiken5, chaiken6, chaiken7} prove that the counting function for the nonattacking placements of $k$ chess pieces with unbounded straight-line moves (also called riders) on a square board of size $m$ is a cyclically repeating sequence of quasi-polynomials $Q(m,k)$ in $m$ of degree $2k$, whose coefficients themselves, after a normalization by $k!$, are polynomials in $k$. The general form for the counting function constituents is, then,
\begin{align*}
    Q(m,k) = \sum_{d=0}^{2k} c_d(k)m^d,
\end{align*}
where the coefficients $c_d(k)$ may vary cyclically, depending on $m$ modulo a number $t$, the period, but not on $m$ itself.

In this paper, we derive explicit formulas for counting the possible nonattacking placements of bishops (two diagonal moves) and anassas (one horizontal or vertical move and one diagonal move) on the square chessboard. The \emph{board} $\mathbf{B}_m$ is defined as the set of coordinate pairs of its squares, and each \emph{piece type} is defined by a set $\mathbf{M}$ of \emph{basic moves}, which are non-zero, non-parallel integral vectors with coprime coordinates. The central ideas of this paper are the following Definition \ref{def:collapsibility} and Theorem \ref{thm:bijection}:
\begin{definition}[Collapsibility, Inductive Subset]\label{def:collapsibility}
  Let $\mathbf{B}_m$ be a board of size $m$, let $\mathbf{I}_m$ be a suitably chosen subset of $\mathbf{B}_m$, and let $\mathbf{M}_P$ be the set of moves of a specific piece type $P$. The board $\mathbf{B}_m$ is said to be \emph{collapsible} under $\{\mathbf{M}_P, \mathbf{I}_m\}$ if there exists a bijection $f:(\mathbf{B}_m-\mathbf{I}_m) \rightarrow \mathbf{B}_{m-1}$ that preserves the neighborhood relations induced by $\mathbf{M}_P$ on the board squares. In such conditions, we also say $\mathbf{I}_m$ is an \emph{inductive subset} of $\mathbf{B}_m$ under $\{f,\mathbf{M}_P\}$.
\end{definition}
\begin{theorem}\label{thm:bijection}
    Let $P_{\mathbf{B}-\mathbf{I}}(m,k)$ and $P_{\mathbf{B}}(m-1,k)$ enumerate the possible nonattacking placements of $k$ pieces $P$, with move set $\mathbf{M}_P$, on the boards $\mathbf{B}_m-\mathbf{I}_m$ and $\mathbf{B}_{m-1}$, respectively, where $\mathbf{I}_m\subset\mathbf{B}_m$ and $\mathbf{B}_{m}$ is a board of size $m$, collapsible under $\{\mathbf{M}_P, \mathbf{I}_m\}$. Then,
    \begin{align*}
        P_{\mathbf{B}-\mathbf{I}}(m,k)=P_{\mathbf{B}}(m-1,k).
    \end{align*}
\end{theorem}
\begin{proof}
    If $\mathbf{B}_{m}$ is collapsible under $\{\mathbf{M}_P, \mathbf{I}_m\}$, then, by Definition \ref{def:collapsibility}, there exists a one to one correspondence between the nonattacking configurations of the $k$ pieces $P$ on boards $\mathbf{B}_m-\mathbf{I}_m$ and $\mathbf{B}_{m-1}$. Hence, the respective enumerating functions must be equal.
\end{proof}
The general strategy used in this paper is to find a specific inductive subset $\mathbf{I}_m$ of $\mathbf{B}_m$ and then analyze how the nonattacking configurations placed on $\mathbf{B}_m-\mathbf{I}_m$ attack $\mathbf{I}_m$. Then, we apply Theorem \ref{thm:bijection} in order to derive a recurrence for $P_{\mathbf{B}}(m,k)$.

In the following sections, we apply these concepts to nonattacking placements of bishops and anassas on the square chessboard. Throughout this paper, unless otherwise specified, we make use of extended definitions for the binomial coefficients \cite{kronenburg} and for the Stirling numbers of both kinds \cite[p.\ 267]{knuth,knuth2}, allowing their arguments to assume any integer value. We also provide links to relevant entries in the On-Line Encyclopedia of Integer Sequences (OEIS) \cite{OEIS} and use the following notation:
\begin{center}
\begin{tabular}{cl}
    $\displaystyle\stirling{m}{k}$ & Stirling number of the second kind (\seqnum{A008277}). \\[4mm]
    $\displaystyle\ustirlingfirst{m}{k}$ & Unsigned Stirling number of the first kind (\seqnum{A132393}).\\[4mm]
    $\displaystyle\Astirling{m}{k}$ & Associated Stirling number of the second kind (\seqnum{A008299}).\\[4mm]
    $\displaystyle\delta_{ij}$ & Kronecker Delta function, 1 if $i=j$ and 0 otherwise. \\[2mm]
    $\displaystyle\pi(m)$ & Parity function, 1 if $m$ odd and 0 otherwise. \\[2mm]
    $\displaystyle[m^d]Q(m,k)$ & Coefficient of $m^d$ in the quasi-polynomial $Q(m,k)$. \\[2mm]
    $\displaystyle\lceil x\rceil$ & Ceiling function, $\min\{n\ |\ n\geq x$, $n$ integer\}. \\[2mm]
    $\displaystyle\lfloor x\rfloor$ & Floor function, $\max\{n\ |\ n\leq x$, $n$ integer\}. \\[2mm]
    $\displaystyle(m)_k$ & Falling factorial, $m(m-1)\cdots(m-k+1)$, $(m)_0=1$.
\end{tabular}
\end{center}

\section{The bishop}
\subsection{Recurrences and closed-form solutions}
We provide an alternative proof for the following result, first obtained by Arshon \cite{arshon} and then refined by Kot\v{e}\v{s}ovec \cite[p.\ 248]{kotesovec}.
\begin{theorem}[Arshon, Kot\v{e}\v{s}ovec]\label{Thm:Arshon, Kotesotev}
  Let $B_{\mathbf{S}}(m,k)$ be the number of possible nonattacking configurations of $k$ bishops placed on the square board $\mathbf{S}_m$. Then,
  \begin{align*}
     B_{\mathbf{S}}(m,k) =\sum_{j=0}^{m}\sum_{i=0}^{\lfloor\frac{m+1}{2}\rfloor}\binom{\left\lfloor \frac{m+1}{2} \right\rfloor}{i}\stirling{i+\left\lfloor\frac{m}{2}\right\rfloor}{m-j}\sum_{l=0}^{\lfloor\frac{m}{2}\rfloor}\binom{\lfloor\frac{m}{2}\rfloor}{l}\stirling{l+\left\lfloor\frac{m+1}{2}\right\rfloor}{m-k+j}.
  \end{align*}
\end{theorem}

Although this formula is correct, Chaiken, Hanusa, and Zaslavsky \cite{chaiken5} state that it is not entirely satisfactory due to the dependency of the summation indices on $m$, which makes its quasi-polynomial nature unclear. We provide a further refinement to this expression that makes the quasi-polynomial form more evident.

First, we divide the square board into white and black squares, noting that a bishop, once placed on one of the colors, never leaves it. Then, the problem reduces to placing bishops on two independent boards. Table \ref{tab:bishops} illustrates these boards and their associated square board. In order to simplify the visualization, the diagonals of the square board $\mathbf{S}_m$ were translated into rows and columns, so that placing rooks on such translated boards corresponds to placing bishops on the respective colors of the square board.
\begin{table}
\begin{center}
\begin{tabular}{c|c|c|c}
      $m$ & Square Board ($\mathbf{S}_m$) & White Board 
      ($\mathbf{W}_m$) & Black Board ($\mathbf{K}_m$)\\
      \hline & & & \\
      0 & $\emptyset$ & $\emptyset$ & $\emptyset$ \\
       & & & \\
      \raisebox{.5\height}{1} & \ytableaushort {\none} *{1} & \ytableaushort {\none} *{1} & \raisebox{.5\height}{$\emptyset$} \\ & & & \\
      \raisebox{-0.5\height}{2} & \ytableaushort {\none} *{2,2} *[*(black)]{1+1,1} &
     \ytableaushort {\none} *{2} &  \ytableaushort {\none} *{1,1}\\ & & & \\
     \raisebox{-2\height}{3} & \ytableaushort {\none} *{3,3,3} *[*(black)]{1+1,1,1+1} *[*(black)]{0,2+1} &
     \ydiagram {1+1, 3, 1+1} & \ytableaushort {\none} *{2,2}\\
     & & & \\
     \raisebox{-3\height}{4} & \ytableaushort {\none} *{4,4,4,4} *[*(black)]{1+1,1,1+1,1} *[*(black)]{3+1,2+1,3+1,2+1} &
     \ydiagram {1+2,4,1+2} & \ydiagram {1+1,3,3,1+1}\\
     & & & \\
     \raisebox{-4\height}{5} & \ytableaushort {\none} *{5,5,5,5,5} *[*(black)]{1+1,1,1+1,1,1+1} *[*(black)]{3+1,2+1,3+1,2+1,3+1} *[*(black)]{0,4+1,0,4+1,0}&
     \ydiagram {2+1,1+3,5,1+3,2+1} & \ydiagram {1+2,4,4,1+2}\\
     & & & \\
     \raisebox{-5\height}{6} & \ytableaushort {\none} *{6,6,6,6,6,6} *[*(black)]{1+1,1,1+1,1,1+1,1} *[*(black)]{3+1,2+1,3+1,2+1,3+1,2+1} *[*(black)]{5+1,4+1,5+1,4+1,5+1,4+1}&
     \ydiagram {2+2,1+4,6,1+4,2+2} & \ydiagram {2+1,1+3,5,5,1+3,2+1} 
  \end{tabular} 
  \end{center}
  \caption{\label{tab:bishops} The sequences of boards $(\mathbf{S}_m)_{m \geq 0}$, $(\mathbf{W}_m)_{m \geq 0}$, and $(\mathbf{K}_m)_{m \geq 0}$ up to $m=6$, corresponding to the square board and the associated white and black boards. The diagonals of $\mathbf{S}_m$ are translated into rows and columns, so that placing bishops on $\mathbf{S}_m$ is equivalent to placing rooks on $\mathbf{W}_m$ and on $\mathbf{K}_m$.}
\end{table}

  From the point of view of Definition \ref{def:collapsibility}, note that the bishop has move set $\mathbf{M}_B=\{(1,1),(-1,1)\}$. We choose an inductive subset candidate $\mathbf{I}_m$ as the shaded squares of Figure \ref{fig:im}, which translate to the boards $\mathbf{W}_m$ and $\mathbf{K}_m$ as the leftmost longest columns $\mathbf{I}_m^\mathbf{W}$ and $\mathbf{I}_m^\mathbf{K}$, respectively (shaded squares of Figures \ref{fig:iwm} and \ref{fig:ikm}). We define a bijection $f_B:(i,j) \in (\mathbf{S}_m-\mathbf{I}_m) \rightarrow (k,l) \in \mathbf{S}_{m-1}$ as follows:
\begin{align*}
    (k,l) = 
    \begin{cases}
      (i-1,j), & \text{if $2 \leq i \leq m$ and $i > j$;}\\
      (i,j-1), & \text{if $3 \leq j \leq m$ and $j > i + 1$.}
    \end{cases}
\end{align*}

Bijection $f_B$ can be interpreted as a one-unit downward translation of the upper part of $\mathbf{S}_m-\mathbf{I}_m$ followed by a one-unit leftward translation of the lower part of $\mathbf{S}_m-\mathbf{I}_m$, in order to build $\mathbf{S}_{m-1}$. Then, it is straightforward to verify that $\mathbf{S}_m$ is collapsible under $\{\mathbf{M}_B, \mathbf{I}_m\}$, because all diagonal neighborhood relations induced by $\mathbf{M}_B$ are preserved by $f_B$. Bijection $f_B$ is equivalent to a horizontal joining of both parts of $\mathbf{W}_m-\mathbf{I}_m^\mathbf{W}$, in order to build $\mathbf{W}_{m-1}$, and similarly for $\mathbf{K}_m-\mathbf{I}_m^\mathbf{K}$, in order to build $\mathbf{K}_{m-1}$.
\begin{figure}[ht]
\centering 
\subcaptionbox{\label{fig:im}}{
  \begin{ytableau}
          \none[6] & & & & & *(gray!40) &  *(gray!40) \\
          \none[5] & & & & *(gray!40)&*(gray!40) & \\
          \none[4] & & & *(gray!40) &  *(gray!40)& &  \\
          \none[3] & &  *(gray!40)&  *(gray!40)& & &  \\
          \none[2] & *(gray!40)& *(gray!40)& & & & \\
          \none[1] & *(gray!40)& & & & &  \\
          \none & \none[1] & \none[2] & \none[3] & \none[4] & \none[5] & \none[6] 
        \end{ytableau}}
        \hspace{2em}
       \subcaptionbox{\label{fig:iwm}}{
  \ydiagram {2+2,1+4,6,1+4,2+2} *[*(gray!40)]{2+1,2+1,2+1,2+1,2+1}
  \vspace{2em}}
  \hspace{2em}
  \subcaptionbox{\label{fig:ikm}}{
  \ydiagram {2+1,1+3,5,5,1+3,2+1} *[*(gray!40)]{2+1,2+1,2+1,2+1,2+1,2+1}
  \vspace{1em}}
  \captionsetup{subrefformat=parens}
\caption{\subref*{fig:im}~Board $\mathbf{S}_6$ with $\mathbf{I}_6$ shaded. \subref*{fig:iwm}~White board $\mathbf{W}_6$ with $\mathbf{I}_6^\mathbf{W}$ shaded. \subref*{fig:ikm}~Black board $\mathbf{K}_6$ with $\mathbf{I}_6^\mathbf{K}$ shaded.
\label{fig:bishopInductiveCol2}}
\end{figure}
\begin{theorem}
\label{thm:whitebishops}
  Let $R_{\mathbf{W}}(m,k)$ be the number of possible nonattacking configurations of $k$ rooks placed on $\mathbf{W}_m$. Similarly, let $R_{\mathbf{K}}(m,k)$ be the number of possible nonattacking configurations of $k$ rooks placed on $\mathbf{K}_m$. Then,
  \begin{align}
     R_{\mathbf{W}}(m,k) &= R_{\mathbf{W}}(m-1,k) + (m-k+\pi(m))R_{\mathbf{W}}(m-1,k-1),
     \label{eq:recurrenceWhitebishops}\\
      R_{\mathbf{K}}(m,k) &= R_{\mathbf{K}}(m-1,k) + (m-k+1-\pi(m))R_{\mathbf{K}}(m-1,k-1)\label{eq:recurrenceBlackbishops},
  \end{align}
  with initial conditions $R_{\mathbf{W}}(m,0)=R_{\mathbf{K}}(m,0)=1$ and $R_{\mathbf{W}}(0,k)=R_{\mathbf{K}}(0,k)=\delta_{k0}$. These recurrences have the following closed-form solutions:
  \begin{align}
     R_{\mathbf{W}}(m,k) &= \sum_{j=0}^{k}\binom{\left\lceil\frac{m}{2}\right\rceil}{j}\stirling{m-j}{m-k},
     \label{eq:closedWhitebishop}\\
          R_{\mathbf{K}}(m,k) &= \sum_{j=0}^{k}\binom{\lfloor\frac{m}{2}\rfloor}{j}\stirling{m-j}{m-k},
     \label{eq:closedBlackbishops}
  \end{align}
  valid for all integers $m,k \geq 0$.
\end{theorem}
\begin{proof}
First, note that $\mathbf{I}_m^\mathbf{W}$ has $m-1+\pi(m)$ squares and each nonattacking rook placed on $\mathbf{W}_m - \mathbf{I}_m^\mathbf{W}$ attacks $\mathbf{I}_m^\mathbf{W}$ on exactly one distinct square. Thus, we have two disjoint possibilities for placing $k$ nonattacking rooks on $\mathbf{W}_m$:
\begin{enumerate}[wide=0.005mm,label=Case \arabic*.]
\item Place all of them on $\mathbf{W}_m - \mathbf{I}_m^\mathbf{W}$. This can be done in $R_{\mathbf{W}-\mathbf{I}^\mathbf{W}}(m,k)=R_{\mathbf{W}}(m-1,k)$ ways, using Theorem \ref{thm:bijection}.
\item Place $k-1$ of them on $\mathbf{W}_m - \mathbf{I}_m^\mathbf{W}$ and then place the remaining rook on one of the $m-1+\pi(m)-(k-1)=m-k+\pi(m)$ unattacked squares of $\mathbf{I}_m^\mathbf{W}$. This can be done in $(m-k+\pi(m))R_{\mathbf{W}-\mathbf{I}^\mathbf{W}}(m,k-1) = (m-k+\pi(m))R_{\mathbf{W}}(m-1,k-1)$ ways, using Theorem \ref{thm:bijection}.
\end{enumerate}

These cases establish Equation \eqref{eq:recurrenceWhitebishops}. The initial conditions are consequences of an empty board, due either to the absence of squares ($m=0$) or to the absence of pieces ($k=0$). In order to prove the validity of the closed-form solution, first check the initial conditions, using the identity $\stirling{-j}{-k}=\ustirlingfirst{k}{j}$, 
\begin{align*}
    R_{\mathbf{W}}(m,0) &= \binom{\left\lceil\frac{m}{2}\right\rceil}{0}\stirling{m}{m} = 1, \\
    R_{\mathbf{W}}(0,k) &= \sum_{j=0}^k\binom{0}{j}\stirling{-j}{-k}=\ustirlingfirst{k}{0}=\delta_{k0}.
\end{align*}
Then, replace \eqref{eq:closedWhitebishop} in \eqref{eq:recurrenceWhitebishops} and use the identities $m-1+\pi(m-1) = m-\pi(m)$, $\stirling{m-j}{m-k}=\stirling{m-1-j}{m-1-k}+(m-k)\stirling{m-1-j}{m-k}$, and $\binom{\frac{1}{2}(m-\pi(m))}{j}+\pi(m)\binom{\frac{1}{2}(m-\pi(m))}{j-1} = \binom{\frac{1}{2}(m+\pi(m))}{j}$ in order to prove that the recurrence holds:
\begin{align*}
    &R_{\mathbf{W}}(m-1,k)+(m-k+\pi(m))R_{\mathbf{W}}(m-1,k-1) =\\ &=\sum_{j=0}^{k}\binom{\frac{1}{2}(m-1+\pi(m-1))}{j}\stirling{m-1-j}{m-1-k} +\\ &+ (m-k+\pi(m))\sum_{j=0}^{k-1}\binom{\frac{1}{2}(m-1+\pi(m-1))}{j}\stirling{m-1-j}{m-k} \\
    &=\binom{\frac{1}{2}(m-\pi(m))}{k}+\sum_{j=0}^{k-1}\binom{\frac{1}{2}(m-\pi(m))}{j}\left[\stirling{m-j}{m-k} + \pi(m)\stirling{m-1-j}{m-k}\right] \\ &=\sum_{j=0}^{k}\binom{\frac{1}{2}(m-\pi(m))}{j}\stirling{m-j}{m-k} + \pi(m)\sum_{j=1}^{k}\binom{\frac{1}{2}(m-\pi(m))}{j-1}\stirling{m-j}{m-k} \\ &= \stirling{m}{m-k}+\sum_{j=1}^{k}\left[\binom{\frac{1}{2}(m-\pi(m))}{j}+\pi(m)\binom{\frac{1}{2}(m-\pi(m))}{j-1}\right]\stirling{m-j}{m-k} \\  &=\sum_{j=0}^{k}\binom{\frac{1}{2}(m+\pi(m))}{j}\stirling{m-j}{m-k} = R_{\mathbf{W}}(m,k).
\end{align*}

For the black board, note that $\mathbf{I}_m^\mathbf{K}$ has $m-\pi(m)$ squares and follow similar steps to obtain \eqref{eq:recurrenceBlackbishops}. Then, use a similar reasoning to check the validity of the closed-form solution \eqref{eq:closedBlackbishops}.
\end{proof}
\begin{remark}
In fact, the approach used to solve recurrence \eqref{eq:recurrenceWhitebishops}, for instance, is more natural if we consider $R_{\mathbf{W}}(m,m-k)$, using the boundary condition $R_{\mathbf{W}}(m,m)=\delta_{m1}+\delta_{m0}$. By solving for some small $k$, the reader may correctly guess and check the following closed-form solution:
\begin{align*}
R_{\mathbf{W}}(m,m-k) = \sum_{j=0}^{k}\binom{k}{j}\frac{(-1)^{k-j}(j+1)^{\frac{1}{2}(m+\pi(m))}j^{\frac{1}{2}(m-\pi(m))}}{k!}.
\end{align*}
This was essentially the expression found by Arshon \cite{arshon}, as stated by Kot\v{e}\v{s}ovec \cite[pp.\ 242--248]{kotesovec}. By expanding $(j+1)^{\frac{1}{2}(m+\pi(m))}$ using the binomial theorem, rearranging summation indices and using the identity $\stirling{m}{k}=\sum_{j=0}^{k}\binom{k}{j}\frac{(-1)^{k-j}j^m}{k!}$, Kot\v{e}\v{s}ovec was able to refine Arshon's result, obtaining
\begin{align*}
R_{\mathbf{W}}(m,k) = \sum_{j=0}^{\frac{1}{2}(m+\pi(m))}\binom{\frac{1}{2}(m+\pi(m))}{j}\stirling{j+\frac{1}{2}(m-\pi(m))}{m-k}.
\end{align*}
However, this can be further simplified by noting that
\begin{align*}
R_{\mathbf{W}}(2m,k) &= \sum_{j=m-k}^{m}\binom{m}{j}\stirling{j+m}{2m-k}=\sum_{j=0}^{k}\binom{m}{j}\stirling{2m-j}{2m-k}, \\
R_{\mathbf{W}}(2m-1,k) &= \sum_{j=m-k}^{m}\binom{m}{j}\stirling{j+m-1}{2m-1-k}= \sum_{j=0}^{k}\binom{m}{j}\stirling{2m-1-j}{2m-1-k}.
\end{align*}
Combining both equations results in \eqref{eq:closedWhitebishop}. A similar approach yields \eqref{eq:closedBlackbishops}.
\end{remark}
\begin{theorem}\label{thm:closedbishops}
  Let $B_{\mathbf{S}}(m,k)$ be the number of possible nonattacking configurations of $k$ bishops placed on the square board $\mathbf{S}_m$. Then,
  \begin{align}
     B_{\mathbf{S}}(m,k)
     &=\sum_{j=0}^{k}\sum_{i=0}^{j}\binom{\lfloor\frac{m}{2}\rfloor}{i}\stirling{m-i}{m-j}\sum_{l=0}^{k-j}\binom{\lceil\frac{m}{2}\rceil}{l}\stirling{m-l}{m-k+j},
     \label{eq:closedbishops}
  \end{align}
  for all integers $m,k \geq 0$.
\end{theorem}
\begin{proof}
Replace Equations \eqref{eq:closedBlackbishops} and \eqref{eq:closedWhitebishop} in the following:
\begin{align}
B_{\mathbf{S}}(m,k) = \sum_{j=0}^{k}R_{\mathbf{K}}(m,j)R_{\mathbf{W}}(m,k-j). \label{eq:bishopsNotDeveloped}
\end{align}
\end{proof}
\par Equation \eqref{eq:closedbishops} makes evident the quasi-polynomial nature of $B_{\mathbf{S}}(m,k)$ and gives an upper bound for the period $t=2$. Regarding the exact period, a deeper analysis is needed because there may be cancellations in the sums. Take, for example, the case $k=1$, $B_{\mathbf{S}}(m,1)=\frac{m^2+\pi(m)}{2}+\frac{m^2-\pi(m)}{2}=m^2$, or the case $k=2$, $B_{\mathbf{S}}(m,2)=12\binom{m}{4}+14\binom{m}{3}+4\binom{m}{2}$, both having period $t=1$. The exact period $t=2$ for all $k \geq 3$ was rigorously proven by Chaiken, Hanusa, and Zaslavsky \cite[Thm.\ 1.1]{chaiken6}.

Theorems \ref{thm:whitebishops} and \ref{thm:closedbishops} led to the submission of a new sequence, \seqnum{A378590}, to the OEIS database, along with updates to the sequences \seqnum{A274105} and \seqnum{A274106}.
\begin{remark}\label{rem:mminusone}
Another aspect mentioned by Chaiken, Hanusa, and Zaslavsky \cite{chaiken5} is the possibility of retrieving the number of combinatorial types of the nonattacking configurations for the piece by making $m=-1$ in the counting formulas. They showed \cite[Thm.\ 5.8]{chaiken1} that this number is $k!$ for all pieces with exactly two unbounded straight-line moves. Indeed, making $m=-1$ in Equation \eqref{eq:closedbishops} yields
\begin{align*}
B_{\mathbf{S}}(-1,k)&=\sum_{j=0}^{k}\sum_{i=0}^{j}\binom{\lfloor-\frac{1}{2}\rfloor}{i}\stirling{-1-i}{-1-j}\sum_{l=0}^{k-j}\binom{\lceil-\frac{1}{2}\rceil}{l}\stirling{-1-l}{-1-k+j}\\
&=\sum_{j=0}^{k}(k-j)!\sum_{i=0}^{j}(-1)^i\ustirlingfirst{j+1}{i+1}=k!,
\end{align*}
where the following identities were used: 
\begin{enumerate*}[label=\alph*)]
\item $\binom{\lfloor-\frac{1}{2}\rfloor}{i}=(-1)^i$; 
\item $\binom{\lceil-\frac{1}{2}\rceil}{l}=\delta_{l0}$; 
\item $\stirling{-m}{-k}=\ustirlingfirst{k}{m}$; 
\item $\ustirlingfirst{k+1}{1}=k!$; and 
\item $\sum_{i=0}^{j}(-1)^i\ustirlingfirst{j+1}{i+1}=\delta_{j0}$.
\end{enumerate*} 
\end{remark}

\subsection{Quasi-polynomial coefficients for the bishop}
In order to proceed to explicit calculation of the quasi-polynomial coefficients for the bishop using Equation \eqref{eq:closedbishops}, two results are needed. 
\begin{lemma} \label{thm:ward}
Let $m$ and $k$ be nonnegative integers. Then,
\begin{align}
\stirling{m}{m-k}=\sum_{p=0}^{k}\Astirling{k+p}{p}\binom{m}{k+p}.
\end{align}
\end{lemma}

An interesting property \cite{fekete} of the numbers $\Astirling{m}{k}$ is that they follow the recurrence
\begin{align*}
\Astirling{m}{k}=k\Astirling{m-1}{k}+(m-1)\Astirling{m-2}{k-1},
\end{align*}
with initial conditions $\Astirling{0}{k}=\delta_{k0}$,  $\Astirling{1}{k}=0$ and $\Astirling{m}{0}=\delta_{m0}$. Written in the form $\Astirling{m+k}{k}$, these numbers are also referenced as the Ward numbers (\seqnum{A134991}). The related OEIS sequences assume $k > 0$ and $m \geq 2$, but here we have trivially extended the definition of $\Astirling{m}{k}$ in order to account for the cases $k=0$ and $0 \leq m<2$ as well. Lemma \ref{thm:ward} is a well-known result and may be proven by induction.
\begin{lemma} \label{thm:binomial}
Let $p$ and $q$ be nonnegative integers, and let $x$ and $z$ be arbitrary complex numbers. Then,
\begin{align}
\binom{2x+z-q}{p}\binom{x}{q}=\sum_{i=0}^{p+q}\beta_i(p,q,z)\binom{2x+z}{i}, \label{eq:binomial}
\end{align}
where
\begin{align}
\beta_i(p,q,z) = \frac{1}{2^{2q}}\sum_{b=\max(0,i-p)}^{q}2^b\binom{2q-b}{q}\sum_{a=0}^{b}\binom{p+b-q-z}{a}\binom{q+z}{b-a}\binom{a-q}{p+b-i}.\label{eq:binomialAlpha}
\end{align}
\end{lemma}
\begin{proof}
We begin with two combinatorial identities, listed by Gould and Quaintance \cite[identities (1.40) and (1.42)]{gould}. Let $x$ be an arbitrary complex number and let $q$ be a nonnegative integer. Then,
\begin{align}
(-1)^{q}2^{2q}\binom{x}{q}=\sum_{b=0}^{q}\binom{-q-1}{q-b}\binom{2q-2x-1}{b}2^{b}.\label{eq:idt}
\end{align}
Now, let $p$ be a nonnegative integer and let $z$ be an arbitrary complex number. Use the identity $\binom{-x}{q}=(-1)^q\binom{x+q-1}{q}$ and multiply $\binom{2x+z-q}{p}$ on both sides of \eqref{eq:idt} to obtain
\begin{align}\label{eq:inter1}
\binom{2x+z-q}{p}\binom{x}{q}=\frac{1}{2^{2q}}\sum_{b=0}^{q}2^{b}\binom{2q-b}{q}\binom{2x+z-q}{p}\binom{2x-2q+b}{b}.
\end{align}
We also have the following result \cite[identity (6.81)]{gould} for arbitrary complex numbers $w$ and $y$ and nonnegative integers $m$ and $n$: 
\begin{align}\label{eq:nanj}
\binom{w}{m}\binom{y}{n}=\sum_{k=0}^{n}\binom{m-w+y}{k}\binom{n+w-y}{n-k}\binom{w+k}{m+n}.
\end{align}
Making $w=2x+z-q$, $m=p$, $y=2x-2q+b$, $n=b$, and expanding the binomial product in \eqref{eq:inter1} yields
\begin{align*}
\binom{2x+z-q}{p}\binom{x}{q}=\frac{1}{2^{2q}}\sum_{b=0}^{q}2^{b}\binom{2q-b}{q}\sum_{a=0}^b\binom{p+b-q-z}{a}\binom{q+z}{b-a}\binom{2x+z-q+a}{p+b}.
\end{align*}
Then, use the generalized version of the Vandermonde Convolution \cite[identity (6.9)]{gould} to write
\begin{align*}
\binom{2x+z-q+a}{p+b}=\sum_{i=0}^{p+b}\binom{2x+z}{i}\binom{a-q}{p+b-i}.
\end{align*}
The final step is to rearrange the summation indices in order to obtain \eqref{eq:binomial} and \eqref{eq:binomialAlpha}.
\end{proof}
\begin{theorem}\label{thm:coeffbishops}
    The quasi-polynomial coefficients of $R_{\mathbf{W}}(m,k)$, $R_{\mathbf{K}}(m,k)$ and $B_{\mathbf{S}}(m,k)$ are, respectively,
    \begin{align}
        [m^d]R_{\mathbf{W}}(m,k)&=\sum_{i=d}^{2k}\frac{(-1)^{i-d}}{i!}\ustirlingfirst{i}{d}\sum_{p=\max(0,i-k)}^{k}\sum_{j=p}^{k}\Astirling{p+j}{p}\beta_i(p+j,k-j,-\pi(m)), \label{eq:coeffWhitebishops} \\
        [m^d]R_{\mathbf{K}}(m,k)&=\sum_{i=d}^{2k}\frac{(-1)^{i-d}}{i!}\ustirlingfirst{i}{d}\sum_{p=\max(0,i-k)}^{k}\sum_{j=p}^{k}\Astirling{p+j}{p}\beta_i(p+j,k-j,\pi(m)), \label{eq:coeffBlackbishops}\\
         [m^d]B_{\mathbf{S}}(m,k) &= \sum_{j=0}^{k}\sum_{i=0}^{d}\left([m^i]R_{\mathbf{W}}(m,j)\right)\left([m^{d-i}]R_{\mathbf{K}}(m,k-j)\right),\label{eq:coeffbishops}
    \end{align}
    for $0 \leq d \leq 2k$.
\end{theorem}
\begin{proof}
     From Equation \eqref{eq:closedWhitebishop} and Lemma \ref{thm:ward}
    \begin{align*}
        R_{\mathbf{W}}(2m,k) &= \sum_{j=0}^{k}\binom{m}{j}\stirling{2m-j}{2m-k} = \sum_{j=0}^{k}\sum_{p=0}^{k-j}\Astirling{p+k-j}{p}\binom{2m-j}{p+k-j}\binom{m}{j}, \\
        R_{\mathbf{W}}(2m-1,k) &= \sum_{j=0}^{k}\binom{m}{j}\stirling{2m-1-j}{2m-1-k}=\sum_{j=0}^{k}\sum_{p=0}^{k-j}\Astirling{p+k-j}{p}\binom{2m-1-j}{p+k-j}\binom{m}{j}.
    \end{align*}
    Now relabel $j$ to $k-j$, use Lemma \ref{thm:binomial}, and rearrange the summation indices to obtain
    \begin{align*}
        R_{\mathbf{W}}(2m,k) &= \sum_{i=0}^{2k}\sum_{p=\max(0,i-k)}^{k}\sum_{j=p}^{k}\Astirling{p+j}{p}\beta_i(p+j,k-j,0)\binom{2m}{i}, \\
        R_{\mathbf{W}}(2m-1,k) &= \sum_{i=0}^{2k}\sum_{p=\max(0,i-k)}^{k}\sum_{j=p}^{k}\Astirling{p+j}{p}\beta_i(p+j,k-j,-1)\binom{2m-1}{i}.
    \end{align*}
    The final step is to use the identity $\binom{x}{i}=\sum_{d=0}^{i}\frac{(-1)^{i-d}}{i!}\ustirlingfirst{i}{d}x^d$ and rearrange the summation indices again. Here, $L(i,k)=\max(0,i-k)$:
    \begin{align*}
        R_{\mathbf{W}}(2m,k) &= \sum_{d=0}^{2k}\sum_{i=d}^{2k}\frac{(-1)^{i-d}}{i!}\ustirlingfirst{i}{d}\sum_{p=L(i,k)}^{k}\sum_{j=p}^{k}\Astirling{p+j}{p}\beta_i(p+j,k-j,0)(2m)^d, \\
        R_{\mathbf{W}}(2m-1,k) &= \sum_{d=0}^{2k}\sum_{i=d}^{2k}\frac{(-1)^{i-d}}{i!}\ustirlingfirst{i}{d}\sum_{p=L(i,k)}^{k}\sum_{j=p}^{k}\Astirling{p+j}{p}\beta_i(p+j,k-j,-1)(2m-1)^d.
    \end{align*}
    Combining both results yields \eqref{eq:coeffWhitebishops}. A similar approach, using $R_{\mathbf{K}}(2m,k)$ and $R_{\mathbf{K}}(2m+1,k)$, results in \eqref{eq:coeffBlackbishops}. A trivial manipulation in \eqref{eq:bishopsNotDeveloped} yields \eqref{eq:coeffbishops}.
\end{proof}

\section{The anassa}
\subsection{Recurrence and closed-form solution}
In this paper, without any loss of generality, anassas are considered to be pieces with move set $\mathbf{M}_A=\{(0,1),(1,1)\}$, as shown in Figure \ref{fig:anassa}. We choose $\mathbf{\Lambda}_m=\mathbf{V}_m\cup\mathbf{D}_m$ (Figure \ref{fig:lambda}) as the inductive subset and define a bijection $f_A:(i,j) \in (\mathbf{S}_m-\mathbf{\Lambda}_m) \rightarrow (k,l) \in \mathbf{S}_{m-1}$ as follows:
\begin{align*}
    (k,l) = 
    \begin{cases}
      (i,j-1), & \text{if $i < j < m$;}\\
      (i,j), & \text{if $j < i < m$.}\\
    \end{cases}
\end{align*}

Bijection $f_A$ can be interpreted as a one-unit downward translation of the upper part of $\mathbf{S}_m-\mathbf{\Lambda}_m$, in order to build $\mathbf{S}_{m-1}$. Then, it is straightforward to verify that $\mathbf{S}_m$ is collapsible under $\{\mathbf{M}_A,\mathbf{\Lambda}_m\}$ because all right diagonal and vertical neighborhood relations induced by $\mathbf{M}_A$ are preserved by $f_A$.
\begin{figure}[ht]
  \centering 
    \subcaptionbox{\label{fig:anassa}}{
    \begin{ytableau}
          \none[6] & & & *(gray!40) & & *(gray!40) &  \\
          \none[5] & & & *(gray!40) & *(gray!40)& & \\
          \none[4] & & & *(gray!40)\bullet& & & \\
          \none[3] & & *(gray!40) &  *(gray!40) & & &  \\
          \none[2] & *(gray!40) & & *(gray!40) & & & \\
          \none[1] & & & *(gray!40) & & & \\
          \none & \none[1] & \none[2] & \none[3] & \none[4] & \none[5] & \none[6] 
    \end{ytableau}}
    \hspace{2em}
    \subcaptionbox{\label{fig:lambda}}{
        \begin{ytableau}
          \none[6] & & & & & & *(gray!40) \\
          \none[5] & & & & & *(gray!40) &*(gray) \\
          \none[4] & & & & *(gray!40)& & *(gray) \\
          \none[3] & & &  *(gray!40) & & & *(gray)  \\
          \none[2] & & *(gray!40)& & & &*(gray) \\
          \none[1] & *(gray!40)& & & & & *(gray) \\
          \none & \none[1] & \none[2] & \none[3] & \none[4] & \none[5] & \none[6] 
        \end{ytableau}}
        \captionsetup{subrefformat=parens}
  \caption{\subref*{fig:anassa}~Illustration of the attacks of an anassa placed on square $(3,4)$. \subref*{fig:lambda}~Sets $\mathbf{\Lambda}_6$ (light and dark shaded squares), $\mathbf{V}_6$ (only dark shaded squares) and  $\mathbf{D}_6$ (only light shaded squares).
  \label{fig:anassasInSquareBoard}}
\end{figure}

Here, the number of attacked squares in $\mathbf{\Lambda}_m$ depends essentially on how many pieces are placed below and above $\mathbf{\Lambda}_m$, since every nonattacking configuration placed on $\mathbf{S}_m-\mathbf{\Lambda}_m$, with $p$ pieces below $\mathbf{\Lambda}_m$ and $k-p$ pieces above $\mathbf{\Lambda}_m$, attacks $\mathbf{\Lambda}_m$ on exactly $2p+(k-p)=k+p$ distinct squares, with $k$ attacked squares in $\mathbf{D}_m$ and $p$ attacked squares in $\mathbf{V}_m$. Thus, a natural choice is to count the nonattacking configurations with exactly $p$ pieces placed strictly below the main diagonal of $\mathbf{S}_m$.
\begin{theorem}
  Let $A_{\mathbf{S}}(m,k,p)$ be the number of possible nonattacking configurations of $k$ anassas placed on the square board $\mathbf{S}_m$ such that exactly $p$ of them are placed strictly below the board's main diagonal. Then,
  \begin{gather}
     A_{\mathbf{S}}(m,k,p) = A_{\mathbf{S}}(m-1,k,p) + (m-k+1)A_{\mathbf{S}}(m-1,k-1,p)+ \nonumber\\+ (m-p)A_{\mathbf{S}}(m-1,k-1,p-1)+(m-p)(m-k+1)A_{\mathbf{S}}(m-1,k-2,p-1),\label{eq:recurrenceanassas}
  \end{gather}
  with initial conditions $A_\mathbf{S}(0,k,p)=\delta_{k0}\delta_{p0}$, $A_\mathbf{S}(m,0,p)=\delta_{p0}$, $A_\mathbf{S}(m,1,p)=\stirling{m+1}{m}\delta_{p0}+\stirling{m}{m-1}\delta_{p1}$ and $A_\mathbf{S}(m,k,0)=\stirling{m+1}{m-k+1}$. This recurrence has the following closed-form solution:
  \begin{align}
     A_{\mathbf{S}}(m,k,p)=\sum_{j=0}^{p}(m-k+j)_j\binom{k-p-1}{j}\binom{k-j}{k-p}\stirling{m+1}{m-k+j+1},
     \label{eq:closedanassas}
  \end{align}
  valid for all integers $m,k,p \geq 0$.
\end{theorem}
\begin{proof}
 We have the following disjoint possibilities for the placement of $k$ nonattacking anassas on $\mathbf{S}_m$ such that $p$ of them are placed strictly below the board's main diagonal:
 \begin{enumerate}[wide=0.005mm,label=Case \arabic*.]
    \item Place all of them on $\mathbf{S}_m-\mathbf{\Lambda}_m$, while keeping $p$ of them below $\mathbf{\Lambda}_m$. This can be done in $A_{\mathbf{S}}(m-1,k,p)$ ways.
    \item Place $k-1$ of them on $\mathbf{S}_m-\mathbf{\Lambda}_m$, while keeping $p$ of them below $\mathbf{\Lambda}_m$, and then place the remaining anassa on one of the $m-k+1$ unattacked squares of $\mathbf{D}_m$. This can be done in $(m-k+1)A_{\mathbf{S}}(m-1,k-1,p)$ ways.
    \item Place $k-1$ of them on $\mathbf{S}_m-\mathbf{\Lambda}_m$, while keeping $p-1$ of them below $\mathbf{\Lambda}_m$, and then place the remaining anassa on one of the $m-1-(p-1)=m-p$ unattacked squares of $\mathbf{V}_m$. This can be done in $(m-p)A_{\mathbf{S}}(m-1,k-1,p-1)$ ways.
    \item Place $k-2$ of them on $\mathbf{S}_m-\mathbf{\Lambda}_m$, while keeping $p-1$ of them below $\mathbf{\Lambda}_m$, then place the first remaining anassa on one of the $m-1-(p-1)=m-p$ unattacked squares of $\mathbf{V}_m$ and then place the other remaining anassa on one of the $m-k+1$ unattacked squares of $\mathbf{D}_m$. This can be done in $(m-p)(m-k+1)A_{\mathbf{S}}(m-1,k-2,p-1)$ ways.
\end{enumerate}
    These cases establish recurrence \eqref{eq:recurrenceanassas}. The first two initial conditions are consequences of an empty board, due either to the absence of squares ($m=0$) or to the absence of pieces ($k=0$). The last two initial conditions are consequences of the well-known result that states that $\stirling{m+1}{m-k+1}$ enumerates the nonattacking configurations of $k$ rooks (or $k$ anassas, due to the board symmetry) placed on a staircase board, in the shape of a right isosceles triangle, with $m$ squares in its longest row. 
    
    In order to prove the validity of the closed-form solution, first check the initial conditions
    \begin{align*}
    A_\mathbf{S}(0,k,p)&=\sum_{j=0}^{p}(-k+j)_j\binom{k-p-1}{j}\binom{k-j}{k-p}\stirling{1}{-k+j+1}\\&=\delta_{k0}\sum_{j=0}^{p}(j)_j\binom{-p-1}{j}\binom{-j}{-p}\stirling{1}{j+1}= \delta_{k0}\delta_{p0},\\
    A_{\mathbf{S}}(m,0,p)&=\sum_{j=0}^{p}(m+j)_j\binom{-p-1}{j}\binom{-j}{-p}\stirling{m+1}{m+j+1}=\delta_{p0},\\
    A_{\mathbf{S}}(m,1,p)&=\sum_{j=0}^{p}(m-1+j)_j\binom{-p}{j}\binom{1-j}{1-p}\stirling{m+1}{m+j}\\
    &=(m-1)_0\binom{-p}{0}\binom{1}{1-p}\stirling{m+1}{m}+(m)_1\binom{-p}{1}\binom{0}{1-p}\stirling{m+1}{m+1}\\
    &=(\delta_{p1}+\delta_{p0})\stirling{m+1}{m}-m\delta_{p1} = \delta_{p0}\stirling{m+1}{m}+\delta_{p1}\stirling{m}{m-1},\\
    A_{\mathbf{S}}(m,k,0)&=\sum_{j=0}^{0}(m-k+j)_j\binom{k-1}{j}\binom{k-j}{k}\stirling{m+1}{m-k+j+1} = \stirling{m+1}{m-k+1}.
    \end{align*}
    Then, in order to replace \eqref{eq:closedanassas} in \eqref{eq:recurrenceanassas}, note that
    \begin{subequations}
    \begin{align}
     &(m-p)A_{\mathbf{S}}(m-1,k-1,p-1)= \nonumber\\&=(m-p)\sum_{j=0}^{p}(m-k+j)_j\binom{k-p-1}{j}\binom{k-j-1}{k-p}\stirling{m}{m-k+j+1},\label{eq:anassasSubs1}\\
     &(m-k+1)A_{\mathbf{S}}(m-1,k-1,p)= \nonumber\\&=(m-k+1)\sum_{j=0}^{p}(m-k+j)_j\binom{k-p-2}{j}\binom{k-j-1}{k-p-1}\stirling{m}{m-k+j+1},\label{eq:anassasSubs2}\\
     &(m-p)(m-k+1)A_{\mathbf{S}}(m-1,k-2,p-1) = \nonumber\\&=(m-p)(m-k+1)\sum_{j=0}^{p-1}(m-k+j+1)_j\binom{k-p-2}{j}\binom{k-j-2}{k-p-1}\stirling{m}{m-k+j+2}\nonumber\\&=(m-p)\sum_{j=1}^{p}(m-k+j)_j\binom{k-p-2}{j-1}\binom{k-j-1}{k-p-1}\stirling{m}{m-k+j+1}.\label{eq:anassasSubs3}
    \end{align}
    \end{subequations}
    Now, sum Equations \eqref{eq:anassasSubs1} to \eqref{eq:anassasSubs3}, simplify the binomial coefficients, and use the recurrence for the Stirling numbers of the second kind to obtain
    \begin{align}\label{eq:intanassa1}
     &(m-p)A_{\mathbf{S}}(m-1,k-1,p-1) + (m-k+1)A_{\mathbf{S}}(m-1,k-1,p) + \nonumber\\&+(m-p)(m-k+1)A_{\mathbf{S}}(m-1,k-2,p-1) =\nonumber\\&=\sum_{j=0}^{p}(m-k+j+1)(m-k+j)_j\binom{k-p-1}{j}\binom{k-j}{k-p}\stirling{m}{m-k+j+1}+\nonumber\\&+\sum_{j=0}^{p}(j+1)(m-k+j)_j\binom{k-p-1}{j+1}\binom{k-j-1}{k-p}\stirling{m}{m-k+j+1}\nonumber\\&=\sum_{j=0}^{p}(m-k+j+1)(m-k+j)_j\binom{k-p-1}{j}\binom{k-j}{k-p}\stirling{m}{m-k+j+1}+\nonumber\\&+\sum_{j=1}^{p}j(m-k+j-1)_{j-1}\binom{k-p-1}{j}\binom{k-j}{k-p}\stirling{m+1}{m-k+j+1}-\nonumber\\&-\sum_{j=1}^{p}j(m-k+j-1)_{j-1}(m-k+j+1)\binom{k-p-1}{j}\binom{k-j}{k-p}\stirling{m}{m-k+j+1}.
    \end{align}
    Then, use the Stirling numbers recurrence again to write
    \begin{align}\label{eq:intanassa2}
        &A_{\mathbf{S}}(m-1,k,p)=\sum_{j=0}^{p}(m-k+j-1)_j\binom{k-p-1}{j}\binom{k-j}{k-p}\stirling{m}{m-k+j}=\nonumber\\
        &=\sum_{j=0}^{p}(m-k+j-1)_j\binom{k-p-1}{j}\binom{k-j}{k-p}\stirling{m+1}{m-k+j+1}-\nonumber\\&-\sum_{j=0}^{p}(m-k+j+1)(m-k+j-1)_j\binom{k-p-1}{j}\binom{k-j}{k-p}\stirling{m}{m-k+j+1}.
    \end{align}
    Now, sum \eqref{eq:intanassa1} and \eqref{eq:intanassa2} and use the identity $(r)_s=(r-1)_s+s(r-1)_{s-1}$ to retrieve $A_{\mathbf{S}}(m,k,p)$, proving that recurrence \eqref{eq:recurrenceanassas} holds.
\end{proof}
\begin{remark}
We may verify that, for $p=k$, Equation \eqref{eq:closedanassas} is a telescoping sum, resulting in
\begin{align*}
&A_{\mathbf{S}}(m,k,k)=\sum_{j=0}^{k}(-1)^j(m-k+j)_j\stirling{m+1}{m-k+j+1}\\&=\sum_{j=0}^{k}(-1)^j\left[(m-k+j)_j\stirling{m}{m-k+j}+(m-k+j+1)_{j+1}\stirling{m}{m-k+j+1}\right]
\\&=\stirling{m}{m-k}.
\end{align*}
This result is expected, since $\stirling{m}{m-k}$ enumerates the nonattacking configurations of $k$ anassas placed on a staircase board of size $m-1$.
\end{remark}
\begin{theorem}\label{thm:closedanassasSummed}
  Let $A_{\mathbf{S}}(m,k)$  be the number of possible nonattacking configurations of $k$ anassas placed on the square board $\mathbf{S}_m$. Then,
  \begin{align}
     A_{\mathbf{S}}(m,k) = \sum_{j=0}^{\lceil\frac{k}{2}\rceil}(m-k+j)_j\stirling{m}{m-k+j}2^{k-2j}\left[\binom{k-j}{j-1}+\binom{k-j+1}{j}\right],\label{eq:closedanassasSummed}
  \end{align}
  for all integers $m,k\geq 0$. 
\end{theorem}
\begin{proof}
First, we use the recurrence for the Stirling numbers of the second kind and manipulate the binomial coefficients to rewrite \eqref{eq:closedanassas} as
\begin{align}\label{eq:intsquareanassas1}
    &A_{\mathbf{S}}(m,k,p)=\sum_{j=0}^{p}(m-k+j)_j\binom{k-p-1}{j}\binom{k-j}{k-p}\stirling{m}{m-k+j}+\nonumber\\&+\sum_{j=1}^{p+1}(m-k+j)_j\binom{k-p-1}{j-1}\binom{k-j+1}{k-p}\stirling{m}{m-k+j}\nonumber\\&=\sum_{j=0}^{p+1}(m-k+j)_j\left[\binom{k-j}{k-p}\binom{k-p}{j}+\binom{k-j}{k-p-1}\binom{k-p-1}{j-1}\right]\stirling{m}{m-k+j}\nonumber\\&=\sum_{j=0}^{p+1}(m-k+j)_j\left[\binom{k-j}{j}\binom{k-2j}{p-j}+\binom{k-j}{j-1}\binom{k-2j+1}{p-j+1}\right]\stirling{m}{m-k+j}.
\end{align}
We also have
\begin{align} \label{eq:intsquareanassas2}
    A_{\mathbf{S}}(m,k) = \sum_{p=0}^{k}A_{\mathbf{S}}(m,k,p).
\end{align}
Then, replace \eqref{eq:intsquareanassas1} in \eqref{eq:intsquareanassas2} and rearrange the summation indices to obtain
\begin{align*}
    A_{\mathbf{S}}(m,k)&=\sum_{j=0}^{k+1}(m-k+j)_j\stirling{m}{m-k+j}\binom{k-j}{j}\sum_{p=j}^{k}\binom{k-2j}{p-j}+\\&+\sum_{j=1}^{k+1}(m-k+j)_j\stirling{m}{m-k+j}\binom{k-j}{j-1}\sum_{p=j-1}^{k}\binom{k-2j+1}{p-j+1}.
\end{align*}
Now, use the identity $\sum_{p=j}^k\binom{k-2j}{p-j}=2^{k-2j}$ and simplify the binomial coefficients. The upper summation limit can be reduced to $\lceil\frac{k}{2}\rceil$, as indicated by the remaining binomial coefficients. These manipulations result in \eqref{eq:closedanassasSummed}.
\end{proof}
\begin{remark}
As done in Remark \ref{rem:mminusone}, making $m=-1$ in \eqref{eq:closedanassasSummed} leads to
\begin{align*}
    A_{\mathbf{S}}(-1,k) &= \sum_{j=0}^{\lceil\frac{k}{2}\rceil}(-1-k+j)_j\stirling{-1}{-1-k+j}2^{k-2j}\left[\binom{k-j}{j-1}+\binom{k-j+1}{j}\right]\nonumber\\&=\sum_{j=0}^{\lceil\frac{k}{2}\rceil}(-1)^j(k)_j(k-j)!2^{k-2j}\left[\binom{k-j}{j-1}+\binom{k-j+1}{j}\right]\nonumber\\&=k!\sum_{j=0}^{\lceil\frac{k}{2}\rceil}(-1)^j2^{k-2j}\left[\binom{k-j}{j-1}+\binom{k-j+1}{j}\right]
    \nonumber\\&=k!\sum_{j=0}^{\lceil\frac{k}{2}\rceil}(-1)^j\left[2^{k+1-2j}\binom{k+1-j}{k+1-2j}-2^{k-2j}\binom{k-j}{k-2j}\right].
\end{align*}
But, by Gould and Quaintance \cite[identity (1.69)]{gould}, we have
\begin{align*}
   \sum_{j=0}^{\lceil\frac{k}{2}\rceil}(-1)^j2^{k-2j}\binom{k-j}{k-2j}=k+1.
\end{align*}
Then, as predicted by Chaiken, Hanusa, and Zaslavsky \cite[Thm.\ 5.8]{chaiken1},
\begin{align*}
A_{\mathbf{S}}(-1,k)=k!(k+2-k-1)=k!.
\end{align*}
\end{remark}
\begin{remark}
Kot\v{e}\v{s}ovec \cite[p.\ 716]{kotesovec} provided two expressions for the anassa (``semi-rook + semi-bishop'') when $k=m$. Here, his expressions are trivially extended to account for the case $m=0$ as well: 
\begin{align*}
    A_{\mathbf{S}}(m,m)=\sum_{j=0}^m\binom{m+1}{j}\stirling{m}{j}\frac{j!}{2^j}=\frac{1}{2^{m+1}}\sum_{j=0}^{m+1}\binom{m+1}{j}j^m.
\end{align*}
Using \eqref{eq:closedanassasSummed}, this result implies the identity
\begin{align*}
    \sum_{j=0}^m\binom{m+1}{j}\stirling{m}{j}\frac{j!}{2^j}=2^m\sum_{j=0}^{\lceil\frac{m}{2}\rceil}\stirling{m}{j}\frac{j!}{2^{2j}}\left[\binom{m-j}{j-1}+\binom{m-j+1}{j}\right].
\end{align*}
However, we could not find a direct simplification for the right-hand side that could prove this result.
\end{remark}
We highlight that Theorem \ref{thm:closedanassasSummed} led to the submission of a new sequence, \seqnum{A378561}, to the OEIS database.

\subsection{Quasi-polynomial coefficients for the anassa}
As for the bishop, an explicit expression for the quasi-polynomial coefficients of $A_{\mathbf{S}}(m,k)$ can be obtained by using Lemma \ref{thm:ward} and some binomial identities.
\begin{theorem}\label{thm:coeffanassas}
    The quasi-polynomial coefficients of $A_{\mathbf{S}}(m,k)$ are
    \begin{align}
        &[m^d]A_{\mathbf{S}}(m,k)=\nonumber\\&=\sum_{i=d}^{2k}\frac{(-1)^{i-d}}{i!}\ustirlingfirst{i}{d}\sum_{j=0}^{U(i,k)}\alpha(k,j)\sum_{p=L(i,k)}^{k-j}\Astirling{p+k-j}{p}\sum_{b=0}^{j}\binom{p}{b}\binom{k}{j-b}\binom{b}{k+p-i},
    \end{align}
    for $0 \leq d \leq 2k$, where $\alpha(k,j)=2^{k-2j}\left[\binom{k-j}{j-1}+\binom{k-j+1}{j}\right]j!$, $U(i,k)=\min(\lceil\frac{k}{2}\rceil,2k-i)$, and $L(i,k)=\max(0,i-k)$.
\end{theorem}
\begin{proof}
    First, use Lemma \ref{thm:ward} to expand the Stirling numbers of the second kind in \eqref{eq:closedanassasSummed}, obtaining
    \begin{align*}
        A_{\mathbf{S}}(m,k)= \sum_{j=0}^{\lceil\frac{k}{2}\rceil}\alpha(k,j)\sum_{p=0}^{k-j}\Astirling{p+k-j}{p}\binom{m}{p+k-j}\binom{m-k+j}{j}.
    \end{align*}
    Then, use identity \eqref{eq:nanj} to expand the binomial product
    \begin{align*}
        A_{\mathbf{S}}(m,k)= \sum_{j=0}^{\lceil\frac{k}{2}\rceil}\alpha(k,j)\sum_{p=0}^{k-j}\Astirling{p+k-j}{p}\sum_{b=0}^{j}\binom{p}{b}\binom{k}{j-b}\binom{m+b}{k+p},
    \end{align*}
    and use the Vandermonde convolution followed by a rearrangement of the summation indices to obtain
        \begin{align*}
        &A_{\mathbf{S}}(m,k)= \sum_{j=0}^{\lceil\frac{k}{2}\rceil}\alpha(k,j)\sum_{p=0}^{k-j}\Astirling{p+k-j}{p}\sum_{b=0}^{j}\binom{p}{b}\binom{k}{j-b}\sum_{i=0}^{k+p}\binom{m}{i}\binom{b}{k+p-i}\\
        &= \sum_{i=0}^{2k}\binom{m}{i}\sum_{j=0}^{\min(\lceil\frac{k}{2}\rceil,2k-i)}\alpha(k,j)\sum_{p=\max(0,i-k)}^{k-j}\Astirling{p+k-j}{p}\sum_{b=0}^{j}\binom{p}{b}\binom{k}{j-b}\binom{b}{k+p-i}.
    \end{align*}
    The final step is to use the identity $\binom{m}{i}=\sum_{d=0}^{i}\frac{(-1)^{i-d}}{i!}\ustirlingfirst{i}{d}m^d$ and rearrange the summation indices again, giving the result.
\end{proof}
\begin{remark}
    Chaiken, Hanusa, and Zaslavsky \cite[section 5.4]{chaiken5} conjectured that all quasi-polynomials for the anassa have a factor of $(m)_k$. This conjecture can be proved true using Equation \eqref{eq:closedanassasSummed} and Lemma \ref{thm:ward}:
    \begin{align*}
        A_{\mathbf{S}}(m,k)= \sum_{j=0}^{\lceil\frac{k}{2}\rceil}2^{k-2j}\left[\binom{k-j}{j-1}+\binom{k-j+1}{j}\right]\sum_{p=0}^{k-j}\Astirling{p+k-j}{p}\frac{(m)_{p+k-j}(m-k+j)_j}{(p+k-j)!}.
    \end{align*}
    It is straightforward to verify that $(m)_{p+k-j}(m-k+j)_j=(m)_k(m-k+j)_p$, hence,
    \begin{align*}
        A_{\mathbf{S}}(m,k)= (m)_k\sum_{j=0}^{\lceil\frac{k}{2}\rceil}2^{k-2j}\left[\binom{k-j}{j-1}+\binom{k-j+1}{j}\right]\sum_{p=0}^{k-j}\Astirling{p+k-j}{p}\frac{(m-k+j)_p}{(p+k-j)!}.
    \end{align*}
\end{remark}

\section{Conclusion}
Theorem \ref{thm:bijection} provides a way to establish recurrences for the number of possible nonattacking placements of bishops and anassas on the square board. The closed-form solutions to these recurrences are the quasi-polynomials \eqref{eq:closedbishops} and \eqref{eq:closedanassasSummed}, for bishops and anassas, respectively. These expressions answer some of the questions raised by Chaiken, Hanusa, and Zaslavsky \cite{chaiken5} and allow the derivation of explicit formulas for the quasi-polynomial coefficients, results proved in Theorems \ref{thm:coeffbishops} and \ref{thm:coeffanassas}. 

\section{Acknowledgments}
I am grateful to God Almighty, for of Him, and through Him, and to Him, are all things.

\bigskip
\hrule
\bigskip

\noindent 2020 {\it Mathematics Subject Classification}:
Primary 05A15; Secondary 00A08.

\noindent \emph{Keywords: } Bishop, Anassa, Nonattacking Placement.

\bigskip
\hrule
\bigskip

\noindent (Concerned with sequences
\seqnum{A008277},
\seqnum{A008299},
\seqnum{A132393},
\seqnum{A134991},
\seqnum{A274105},
\seqnum{A274106},
\seqnum{A378561},
\seqnum{A378590}.)

\bigskip
\hrule
\bigskip

\vspace*{+.1in}
\noindent
Received XXXXXX XX 20XX;
revised versions received XXXXXX XX 20XX; XXXXXX XX 20XX.
Published in {\it Journal of Integer Sequences}, XXXXXX XX 20XX.

\bigskip
\hrule
\bigskip

\noindent
Return to
\href{https://cs.uwaterloo.ca/journals/JIS/}{Journal of Integer Sequences home page}.
\vskip .1in

\end{document}